\documentclass{aptpub}

\usepackage{dsfont,pifont}
\usepackage[T1]{fontenc}
\usepackage[dvips]{graphics}
\usepackage{amsmath}
\usepackage{amssymb}
\usepackage{graphicx}
\usepackage{epsfig}

\newcommand{\vc}[1]{\boldsymbol{#1}}

\newcommand{\calz}{\mathcal Z}

\newcommand{\cqfd}{\hfill $\square$}
\newcommand{\spc}{\textrm{sp}}
\newcommand{\bs}{\boldsymbol}
\usepackage{setspace}
\authornames{S. Hautphenne, G. Latouche, and G. T. Nguyen} 
\shorttitle{Branching processes with infinitely many types} 

\begin{document}

\title{Extinction probabilities of branching processes with countably infinitely many types} 

\authorone[The University of Melbourne]{S. Hautphenne} 
\authortwo[Universit\'e libre de Bruxelles]{G. Latouche} 

\authorthree[The University of Adelaide]{G. T. Nguyen} 

\addressone{Department of Mathematics and Statistics, The University of Melbourne, VIC 3010, Australia.  sophiemh@unimelb.edu.au.} 

\addresstwo{D\' epartement d'Informatique, Universit\'{e} libre de Bruxelles 1050 Brussels, Belgium. latouche@ulb.ac.be.}

\addressthree{School of Mathematical Sciences, The University of Adelaide, 5005, Australia. \newline giang.nguyen@adelaide.edu.au.}

\begin{abstract}
We present two iterative methods for computing the global and partial extinction probability vectors for Galton-Watson processes with countably infinitely many types. The probabilistic interpretation of these methods involves truncated Galton-Watson processes with finite sets of types and modified progeny generating functions. In addition, we discuss the connection of  the convergence norm of the mean progeny matrix with extinction criteria. Finally, we give a sufficient condition for a population to become extinct almost surely even though its population size explodes on the average, which is impossible in a branching process with finitely many types. We conclude with some numerical illustrations for our algorithmic methods.
\end{abstract}



\section{Introduction} 

Branching processes are powerful mathematical tools frequently used to study the evolution of collections of individuals over time. In particular, multi-type Galton-Watson processes represent populations in which individuals are classified into different categories and live for one unit of time. Each individual may reproduce at the end of its lifetime, with reproduction rules  dependent on its type. 

When the number of types is finite, one extinction criterion is based on the spectral radius $\spc(M)$ of the mean progeny matrix $M$, the elements $M_{ij}$ of which are the expected number of direct offsprings with type $j$ for a parent of type~$i$. 
Moreover, the extinction probability vector $\bs{q}$ is the minimal nonnegative solution of the fixed-point equation $\bs{q} = \bs P(\bs{q})$, where each component $q_i$ is the extinction probability given the initial type~$i$, and $\bs P(\cdot)$ is the progeny generating function of the process. Harris~\cite{harris63} and references therein present a comprehensive analysis of  extinction criteria and extinction probability for Galton-Watson processes with finitely many types.

To allow, as we do here, the set of types to be infinite gives rise to three main challenges. First, as the mean progeny matrix $M$ has infinite dimension, one has to look for a replacement to the spectral radius as an extinction criterion. Second, one needs to determine how to compute the extinction probability vector $\bs{q}$ which now has infinitely many entries. Third, the concept of extinction has to be defined carefully: when there are infinitely many types, it is possible for every type to eventually disappear while the whole population itself explodes.  We use the term \emph{global extinction} to indicate that the whole population becomes extinct, and represent by $\bs q$ the probability vector for this event; we refer to the event that every type becomes extinct as \emph{partial extinction}, and denote its probability vector by $\tilde{\bs q}$, with $\bs q \leq \tilde{\bs q}$, naturally, and the question is whether they are equal or not.

Galton-Watson processes with infinitely many types have been much investigated already.  Moyal~\cite{moyal62} assumes that the types belong to an abstract space and proves that the extinction probability is a solution of the fixed point equation $\bs s = \bs{P}(\bs s)$.  Mode~\cite[Theorem~7.2]{mode71}, for a restricted family of progeny densities, gives an extinction criterion based on the spectral radius of some integral operator.  Focusing on denumerably infinite sets of types,  Moy~\cite{moy66,moy67} and Spataru~\cite{spataru89} use ergodic properties for infinite matrices, and analyse in special cases the role of the \emph{convergence norm} of $M$ as an extinction criterion. Recently, some authors in the literature of branching random walks have defined \emph{local survival}, meaning that for every given type $i$ and arbitrarily large epoch $T$ there is at least one individual of  type $i$ alive at some time $t > T$, with \emph{global survival} meaning that at least one individual is alive at any time, and \emph{strong local survival}, when the two have the same  probability.  We refer to Bertacchi and Zucca~\cite{bertacchi09},  Zucca~\cite{zucca11}, and to~Gantert \emph{et al}.~\cite{gantert10}.

There is, however, no simple general extinction criteria for Galton-Watson processes with countably infinitely many types so far, and the question of actually computing the extinction probability vector has received scant attention, if any.

Our main result is the development of two algorithmic methods for computing the global and the local extinction probability vectors $\bs{q}$ and $\tilde{\bs{q}}$. The methods, which are presented in Section~\ref{sec:algos},
have a physical interpretation based on two truncated Galton-Watson processes with finite sets of types.
They may be applied to both irreducible and reducible branching processes with countably infinitely many types.

In Section~\ref{sec:superc} we discuss some extinction criteria expressed in terms of the convergence norm of the  mean progeny matrix $M$ in the irreducible case, or of irreducible sub-matrices of $M$ when $M$ is reducible.  We also give a sufficient condition under which the population becomes extinct almost surely while its expected size tends to infinity. That condition implies that the asymptotic growth rate of the process may depend on the distribution of the initial individual's type. 

In Section~\ref{sec:rand}, we provide some numerical illustrations.  Our examples are taken from two classes of processes for which the matrix $M$ is  tridiagonal (and irreducible) or super-diagonal (and reducible).

\section{Preliminaries} \label{sec:model} 

Consider the process $\{\mathcal{Z}_n=(Z_{n1}, Z_{n2}, \ldots)\}_{n\in
  \mathds{N}}$, where $Z_{n{\ell}}$ is the number of individuals of
type $\ell$ alive at the $n$th generation,
for $\ell$ in the countably infinite set of types $\mathcal{S}= \{1,
2, 3, \ldots\}$.
%
Unless otherwise stated, the process starts in generation~0 with one individual. 

We denote by $p_{i\bs{j}}$ for $\bs{j} = (j_1, j_2, \ldots)$ the
probability that an individual of type~$i$ gives birth to~$j_1$
children of type~1, $j_2$ children of type~2, etc.,
and the progeny generating function
$P_i(\boldsymbol{s})$ of type $i\in \mathcal{S}$ is given by
\begin{align*}
P_i(\bs{s}) = \sum_{\bs j \in \mathds{N}^{\infty}}p_{i\bs{j}} \bs{s}^{\bs{j}} = \sum_{\bs j \in \mathds{N}^{\infty}}p_{i\bs{j}} \prod_{k = 1}^{\infty} s_k^{j_k},
\end{align*}  
with $\bs{s} = (s_1, s_2, \ldots)$, $s_i \in [0,1]$ for all $i$.  We  define $\bs P(\bs{s}) = (P_1(\bs{s}), P_2(\bs{s}), \ldots)$. The {mean progeny matrix} $M$ is defined by
\[
M_{ij} = \left.\frac{\partial P_i(\bs{s})}{\partial s_j} \right |_{\bs{s} = \bs{1}}  \qquad \mbox{for $i,j \in \mathcal S$},
\]
and $M_{ij}$ is the expected number of direct offspring of type $j$ born to a parent of type~$i$. The process $\{\mathcal{Z}_n\}$ is said to be {irreducible} if $M$ is irreducible, and it is {reducible} otherwise.  

The total population size at the $n$th generation is $|\mathcal{Z}_{n}| = \sum_{\ell = 1}^{\infty} {Z}_{n\ell}$, and we denote by $\varphi_0$ the type of the first individual in generation~0.  The conditional {\em global} extinction probability vector, given the initial type, is $\vc{q} = (q_1,q_2, \ldots)$ where
\begin{align*} 
q_{i} & =   \mathds{P}[\lim_{n\rightarrow \infty}\left. |\mathcal{Z}_{n} |= {0} \right| \varphi_0 = {i}] \qquad \mbox{ for } i \in \mathcal{S}.
\end{align*} 
This is the usual conditional probability that the whole population eventually becomes extinct, given the type of the initial individual, and we write that 
$\vc q = \mathds{P}[\lim_{n\rightarrow \infty} \left. |\mathcal{Z}_{n}| = {0} \right| \varphi_0]$ for short.  The vector $\vc q$ is the minimal nonnegative solution of the fixed-point equation
\begin{align}\label{exteq}
\vc{P}(\vc{s}) = \vc{s}.
\end{align}
This equation has at most two distinct solutions, $\vc 1$ and $\vc q\leq \vc 1$, if $M$ is irreducible, and potentially infinitely many solutions otherwise (Moyal~\cite{moyal62}, Spataru~\cite{spataru89}).

The conditional \emph{partial} extinction probability, given the initial type, is $\tilde{\vc{q}} = (\tilde{q}_1, \tilde{q}_2, \ldots) $ where
\begin{align*} 
\tilde{q}_{i} & =  \mathds{P}[\forall \ell: \lim_{n\rightarrow \infty} \left. {Z}_{n\ell} = {0} \right| \varphi_0 = {i}] \qquad \mbox{ for } i \in \mathcal{S}. 
\end{align*}
In the irreducible case, Zucca~\cite{zucca11} observes that  $\lim_{n\rightarrow \infty}  {Z}_{n\ell} = {0}$ for all types~$\ell$ if and only if the limit is zero for at least one type, regardless of the initial type.

The  vector $\tilde{\vc q}$  is also a solution of \eqref{exteq}. Indeed, by conditioning on the progeny of the initial individual and using the independence between individuals, we readily obtain, for any $i$,
\begin{align*}
\tilde{q}_i & = \mathds{P}[\forall \ell: \lim_{n\rightarrow \infty} \left. {Z}_{n\ell} = {0} \right| \varphi_0 = {i}] \\
 & = \sum_{\bs j=(j_1,j_2,\ldots) }p_{i\bs{j}} \prod_{k = 1}^{\infty} \mathds{P}[\forall \ell: \lim_{n\rightarrow \infty} \left. {Z}_{n\ell} = {0} \right| \varphi_0 = {k}]^{j_k} \\
 & = P_i(\tilde{\vc{q}}).
 \end{align*}
    
When the set of types is finite, global and partial extinction are equivalent, but this is not necessarily the case when the set of types is infinite: by Fatou's Lemma:
\begin{align*}
\lim_{n\rightarrow \infty} |\mathcal{Z}_{n} |=\lim_{n\rightarrow \infty} \sum_{\ell=1}^{\infty}  {Z}_{n\ell} \geq \sum_{\ell=1}^{\infty}  \lim_{n\rightarrow \infty}{Z}_{n\ell},
\end{align*}
so that 
\begin{align*} 
\mathds{P}[\lim_{n\rightarrow \infty} |\mathcal{Z}_{n} |=0| \varphi_0 = i] & \leq \mathds{P}[\forall \ell: \lim_{n\rightarrow \infty}{Z}_{n\ell}=0| \varphi_0 = i],
\end{align*}
 for $i,\ell \in \mathcal{S}$
and $\vc0\leq \vc{q}\leq \tilde{\vc{q}} \leq \vc 1$. 

As the vectors $\vc q$, $\tilde{\vc q}$ and $\vc 1$ are all solutions of (\ref{exteq}) and since there are at most two distinct solutions in the irreducible case, the following lemma is immediate.

\begin{lem}\label{lem1} If $M$ is irreducible, and $\tilde{\vc{q}}<\vc 1$, then  $\vc q=\tilde{\vc{q}}$.\cqfd \end{lem}
In the irreducible case, $\vc q = \tilde{\vc q}<\vc1$ is equivalent to strong local survival in the terminology of branching random walks and, although it is expressed differently, this sufficient condition is observed in~\cite{gantert10} with the assumption that $M$ is tridiagonal, and in \cite{zucca11} for the general case.  When $M$ is reducible, it is possible that $\vc{q}< \tilde{\vc{q}} < \vc1$, and we give an example at the end of Section~\ref{sec:rand}.

\section{Computational aspects} \label{sec:algos} 

In this section, we develop iterative methods to compute the extinction probability vectors $\vc q$ and $\tilde{\vc{q}}$. The procedures apply for both irreducible and reducible processes.  The underlying idea is to compute approximations of the infinite vectors $\vc q$ and $\tilde{\vc{q}}$ by solving finite systems of equations in such a way that the successive approximations have probabilistic interpretations: for $\vc q$ we use a time-truncation argument, and for $\tilde{\vc q}$ a space-truncation argument.

\subsection{Global extinction probability}
  \label{sec:global}
Denote by $N_e$ the generation number when the process becomes extinct.  Clearly, 
$\vc{q} = \mathds{P}[{N_e} <\infty\,|\varphi_0]$.
Let $T_k=\{k + 1, k + 2, \ldots\}$ be the set of types strictly greater than $k$, and define the first passage time
$\tau_k=\inf\{n: \sum_{\ell\in T_k}Z_{n\ell}>0\}$, for $k \geq 0$.  This is the first generation when an individual of any type in $T_k$ is born.   Furthermore, define 
\begin{align*}
{q_i^{(k)}} = \mathds{P}[{{N_e}} < {\tau_k}|\varphi_0 = i],
\end{align*}
the conditional probability that the process eventually becomes extinct before the birth of any individual of a type in $T_k$, given that the initial individual has type $i$, and $\bs{q}^{(k)} = (q_1^{(k)}, q_2^{(k)}, \ldots ).$ 

\begin{lem} 
The sequence $\{\bs{q}^{(k)}\}_{k\geq0}$ is monotonically increasing and converges pointwise to the global extinction probability vector $\vc q$.
\end{lem}

\begin{proof}
Clearly, $T_k\supset T_{k+1}$ for all $k$, and  $\tau_k\leq\tau_{k+1}$, so that  $[{N_e} < {\tau_k}] \subseteq [{N_e} < {\tau_{k+1}}]$, and ${\bs{q}^{(k)}}\leq{\bs{q}^{(k+1)}}$.  Therefore, for any $i$,
\begin{align*} 
\lim_{k\rightarrow\infty}q_i^{(k)}& = \lim_{k\rightarrow\infty}\mathds{P}[{{N_e}}<{\tau_k}\,|\,\varphi_0 = i] \\
& = \mathds{P}[{{N_e}}<\lim_{k\rightarrow\infty} {\tau_k} \,|\,\varphi_0 = i]\\
& = \mathds{P}[{{N_e}} < \infty\,|\,\varphi_0 = i] \\
& = q_i, 
\end{align*}
which completes the proof. 
\end{proof}

By definition, ${q}^{(k)}_i=0$ for all $i\in T_k$. Consequently,
 \begin{align*}
\bs{q}^{(0)} & =(0, 0,\ldots), \\
\bs{q}^{(k)} & = (q^{(k)}_1, \ldots, q^{(k)}_k, 0, 0, \ldots) \quad \mbox{ for } k\geq1.
\end{align*}
Thus, at the $k$th iteration one only needs to compute the finite vector $\bs{w}^{(k)} = (q^{(k)}_1, \ldots, q^{(k)}_k)$, which we do as follows. Consider a branching process $\{\mathcal{W}_{n}^{(k)}\}_{n\in\mathds{N}}$ which evolves like $\{\mathcal Z_n\}$ under taboo of the types in $T_{k}$.  The taboo progeny distribution $f_{i\vc j}^{(k)}$ associated with types $i\in\{1, \ldots, k\}$ in $\{\mathcal{W}_{n}^{(k)}\}$ is defined as
\begin{align*} 
f^{(k)}_{i(j_1,\ldots,j_k)}=  
 p_{i(j_1,\ldots,j_k,0,0,\ldots)}.
\end{align*}
If the process is irreducible, then 
$
\sum_{\vc j\in\mathds{N}^{k}} f^{(k)}_{i(j_1,\ldots,j_k)} \leq 1,
$
for $1 \leq i \leq k$, with at least one strict inequality, and we need to add an absorbing state $\Delta$ to the state space $\mathds{N}^{k}$ of $\{\mathcal{W}_{n}^{(k)}\}$ to account for the missing probability mass. Obviously, absorption in $\Delta$ precludes extinction, and $\bs{w}^{(k)}$ is the vector of probability that $\{\mathcal{W}_{n}^{(k)}\}$ becomes extinct before being absorbed in $\Delta$, given the type of the initial particle.  In consequence, $\bs{w}^{(k)}$ is the minimal nonnegative solution of the finite system of equations 
\begin{align} \label{sys1}
 s_i=F^{(k)}_i(s_1,s_2,\ldots,s_k), \qquad \mbox{ for } 1\leq i\leq k,
\end{align} 
where $F_i^{(k)}(\vc s)=P_i(s_1,\ldots,s_k,0,0\ldots)$ is the probability generating function of the distribution $f^{(k)}_{i\vc j}$.

We may compute $\bs{w}^{(k)}$ by linear functional iteration, for instance, on the fixed-point equation \eqref{sys1}: one easily verifies that, for any $k\geq1$, the sequence $\{\bs{w}^{(k,n)}=({w}_1^{(k,n)},\ldots,{w}_k^{(k,n)})\}_{n\geq0}$ recursively defined as
\begin{align*}
w_i^{(k,n)} & = F^{(k)}_i(w_1^{(k,n-1)}, \ldots, w_k^{(k,n-1)}) \qquad \mbox{for $1\leq i\leq k$,} 
\end{align*}
for $n \geq 1$, with $\bs{w}^{(k,0)}=(0,0\ldots,0)$, satisfies
\begin{align*}
\bs{w}^{(k,n)} & = \mathds{P}[{{N_e}} < \tau_k \mbox{ and } N_e \leq n  \,|\; \varphi_0],
\end{align*}
and is, therefore, monotonically increasing to $\bs w^{(k)}$.  In practice, we would  terminate the functional iteration for a given $k$ when $|| \vc w^{(k,n+1)}-\vc w^{(k,n)}||$ becomes sufficiently small.

\subsection{Partial extinction probability}
	\label{subsec:localex}

Here, we associate to the branching process $\{ \calz_n\}$ a family of processes 
$\{\mathcal{Z}^{(k)}_n=(Z^{(k)}_{n1},Z^{(k)}_{n2},\ldots)\}_{n\in\mathds{N}}$, for $k \geq 0$,  obtained as follows: for a given $k$, we count neither the individuals of types in ${T_k}$, nor {\em any} of their descendants, whatever their types. It is as if all individuals of types in ${T_k}$ became {sterile}.
  Define $\tilde{\bs{q}}^{(k)}$ to be the global extinction probability vector of the process $\{\mathcal{Z}^{(k)}_n\}$.

\begin{lem} \label{qtidleseq}
The sequence of vectors $\{\tilde{\bs{q}}^{(k)}\}_{k \geq 0}$ is monotonically decreasing and converges pointwise to the partial extinction probability vector $\tilde{\bs{q}}$.\end{lem}

\begin{proof}
Obviously,
 \begin{align*}\mathcal{Z}^{(k)}_n & =(Z^{(k)}_{n1},Z^{(k)}_{n2},\ldots,Z^{(k)}_{nk},0,0,0,\ldots)\\
            & \leq (Z^{(k+1)}_{n1},Z^{(k+1)}_{n2},\ldots,Z^{(k+1)}_{nk},Z^{(k+1)}_{n(k+1)},0,0,\ldots) \quad \textrm{a.s.}\\
            & = \mathcal{Z}^{(k+1)}_n,
 \end{align*} 
so that, for $n$ fixed and $k \rightarrow \infty$, $\mathcal{Z}^{(k)}_n$ monotonically converges to  $\mathcal{Z}_n$. Furthermore, 
\begin{align*}
\lim_{n\rightarrow\infty} |\mathcal{Z}^{(k)}_n|\leq \lim_{n\rightarrow \infty} |\mathcal{Z}^{(k+1)}_n|,
\end{align*} 
so that  $\tilde{\bs{q}}^{(k)} \geq \tilde{\bs{q}}^{(k+1)}$, for $k\geq0$.

Let $B_{\ell}^{(k)}=[\lim_{n\rightarrow\infty}  Z^{(k)}_{n\ell}=0]$ be the event that type $\ell$ of $\{\mathcal{Z}_{n}^{(k)}\}$ eventually becomes extinct, and let $A^{(k)}=\cap_{\ell\geq 1} B_{\ell}^{(k)}$ be the event that all types of $\{\mathcal{Z}_{n}^{(k)}\}$ eventually become extinct. We have
\begin{align*} \tilde{\bs{q}}^{(k)} &= \mathds{P}[\lim_{n\rightarrow\infty} |\mathcal{Z}^{(k)}_n|=0\,|\,\varphi_0] = \mathds{P}[A^{(k)}\,|\,\varphi_0],\end{align*} 
since $|\mathcal{Z}^{(k)}_n|=\sum_{\ell} Z^{(k)}_{n,\ell}$ contains only finitely many nonzero terms. 
Furthermore, $B_{\ell}^{(k+1)}\subseteq B_{\ell}^{(k)}$, so that $ A^{(k+1)}\subseteq A^{(k)}$, and 
\begin{align*} 
A^{(k)} \searrow A^{(\infty)}=\bigcap_{\ell\geq 1}[\lim_{n\rightarrow\infty}  Z_{n\ell}=0].
\end{align*}
Therefore,
\begin{align*} 
\lim_{k\rightarrow\infty}\tilde{\bs{q}}^{(k)}  = \mathds{P}[A_{\infty}\,|\,\varphi_0] = \mathds{P}[\forall\ell: \lim_{n\rightarrow\infty}  Z_{n\ell}=0\,|\,\varphi_0] = \tilde{\bs{q}},\end{align*}
which completes the proof.
\end{proof}

By definition of $\{\calz_n^{(k)}\}$, $\tilde{q}^{(k)}_i=1$ for all $i\in T_k$, and so
\begin{align*} 
\tilde{\bs{q}}^{(0)} & =  (1,1,\ldots), \\
\tilde{\bs{q}}^{(k)} & = (\tilde{q}^{(k)}_1, \ldots, \tilde{q}^{(k)}_k, 1, 1, \ldots).
\end{align*} 
To compute the finite vector $\widetilde{\bs{w}}^{(k)} = (\tilde{q}^{(k)}_1, \ldots, \tilde{q}^{(k)}_k)$, we may interpret $\{\calz_n^{(k)}\}$ as a Galton-Watson process with finitely many types and progeny distribution 
 $\tilde f_{i\vc j}^{(k)}$  defined as follows:
\begin{align*} 
\tilde f^{(k)}_{i(j_1,\ldots,j_k)}=  \sum_{j_{k+1}, j_{k+2}, \ldots \geq 0}
 p_{i(j_1,\ldots,j_k,j_{k+1}, j_{k+2},\ldots)}.
\end{align*}
and  $\widetilde{\bs{w}}^{(k)} $ is the minimal nonnegative solution of the finite system of equations 
\begin{align*} 
 s_i=\widetilde{F}^{(k)}_i(s_1,s_2,\ldots,s_k), \qquad \mbox{ for } 1\leq i\leq k,
\end{align*} 
where $\widetilde{F}^{(k)}_i(\bs s)=P_i(s_1,\ldots,s_k,1,1,\ldots)$ is the probability generating function of the distribution $\tilde f_{i\vc j}^{(k)}$.  
That equation may be solved by functional iteration, as explained at the end of Subsection~\ref{sec:global}.

\section{Extinction criteria} 
	\label{sec:superc} 

When the number of types is finite and $M$ is irreducible, it is well-known that
\begin{itemize}
\item $\vc q < \vc 1$ if $\spc(M) > 1$,
\item $\vc q = \vc 1$ if $\spc(M) \leq 1$.
\end{itemize}
If $M$ is reducible, then
\begin{itemize}
\item $\vc{q} \lneqq\vc 1$ if and only if  $\spc(M) > 1$,
\end{itemize}
where we write $\vc{v} \lneqq\vc 1$ to indicate that $v_i \leq 1$ for all $i$, with at least one strict inequality.  Indeed, if $M$ is reducible, there may exist some type (but not all) from which partial extinction is almost sure even if  $\spc(M)> 1$  (Hautphenne~\cite{hautphenne12}). 

We  expect that, in the infinite countable case, some analogue of  $\spc(M)$ also plays a role in determining if extinction occurs almost surely or not. This is the case for partial extinction,  but not necessarily for global extinction.

\subsection{Partial extinction --- $M$ irreducible}
	\label{subset:partialcriteria}

We denote by $\widetilde{M}^{(k)}$ the mean progeny matrix of the process $\{\calz_{n}^{(k)}\}$ defined in Section~\ref{subsec:localex}, and by $M^{(k)}$ the $k \times k$ north-west truncations of the infinite matrix $M$. As we do not count individuals with types in $T_k$,  $\widetilde{M}^{(k)}$ is given by
\begin{align*}
\widetilde{M}^{(k)}=\left[\begin{array}{c|c} M^{(k)} & 0 \\\hline  0 & 0\end{array}\right],
\end{align*}
and it is clear that $\tilde{\vc{q}}^{(k)}=
(\tilde{q}_1^{(k)}, \ldots, \tilde{q}_k^{(k)},1,1,\ldots) = \vc 1$ 
if $\spc(M^{(k)}) \leq 1$, otherwise $\tilde{\vc q}^{(k)} \lneqq\vc 1 $, 
with $(\tilde{q}_1^{(k)},\ldots,\tilde{q}_k^{(k)}) < \vc 1 $ if $M^{(k)}$ is irreducible.

We assume in this subsection that  $M$ is irreducible.  The \textit{convergence norm} of $M$ is defined as follows. Let $R$ be the convergence radius of the power series $\sum_{k \geq 0} r^k(M^k)_{ij}$, which does not depend on $i$ and $j$.  The {convergence norm} ${\nu}$ of ${M}$ is 
\begin{align*} 
{\nu} & = {R}^{-1} = {\lim_{k \rightarrow \infty} \{(M^k)_{ij}\}^{1/k}};
\end{align*}
it is also the smallest value such that there exists a nonnegative vector $\vc x$ satisfying $ \vc x M\leq \nu\vc x$ (Seneta~\cite[Definition~6.3]{senetabook}). Note that the convergence norm of a finite matrix is equal to its spectral radius.

If we assume that all but at most a finite number of truncations $M^{(k)}$ are irreducible, then by Seneta (\cite[Theorem 6.8]{senetabook}) the sequence $\{\spc(M^{(k)})\}$  is non-decreasing and $\lim_{k \rightarrow \infty} \spc(M^{(k)}) = \nu$, and one immediately shows the following property.

\begin{prop}\label{locextcrit}  Assume that $M$ is irreducible.  The partial extinction probability vector ${\tilde{\bs{q}}}$ is such that $ {\tilde{\bs{q}}} <\vc 1$ if and only if $\nu>1$.\end{prop}
\begin{proof}
If $\nu > 1$, as  $\spc(M^{(k)})\nearrow \nu$, there exists some $k$ such that $\spc(M^{(k)})>1$ and such that $M^{(k)}$ is irreducible. Thus, $(\tilde{q}_1^{(k)},\ldots,\tilde{q}_k^{(k)})< \vc 1$ and, since $\tilde{\vc q}^{(k)}\searrow {\tilde{\bs{q}}}$ by Proposition \ref{qtidleseq},  $\tilde{\bs{q}} <\vc 1$ in the limit. 

If $\nu\leq 1$, then for all $k$, $\spc(M^{(k)})\leq1$ and $ {\tilde{\bs{q}}}^{(k)} =\vc 1$, which implies that ${\tilde{\bs{q}}} =\vc 1$.
\end{proof}

This is also observed in Gantert \textit{et al.} \cite{gantert10} and Müller \cite{muller08}, who use  different arguments.

\subsection{Partial extinction --- $M$ reducible}
   \label{s:mreducible}

Let us assume now that the matrix $M$ is reducible.
The sequence $\{sp(M^{(k)})\}$ is still non-decreasing, but its limit might not be the convergence norm of~$M$.  Let $\bar{\nu}\in[0,\infty]$ denote the limit. The proof of the  proposition below is very  similar to that of Proposition~\ref{locextcrit} and is omitted.
\begin{prop}\label{locextcrit2}  The partial extinction probability vector ${\tilde{\bs{q}}}$ is such that $ {\tilde{\bs{q}}} =\vc 1$ if and only if $\bar{\nu}\leq1$, otherwise $\tilde{\bs{q}} \lneqq\vc 1$.
\end{prop}  

In other words, there exists at least one type $i$ such that $\tilde{q}_i<1$ if and only if $\bar{\nu}>1$. The next question is, for which $i$ does the inequality $\tilde{q}_i<1$ hold? We give below a necessary and sufficient condition for $\tilde{q}_i$ to be strictly less than one.

We write $i\rightarrow j$ when type $i$ has a positive probability to generate an individual of type $j$ in a subsequent generation, that is, if there exists $n\geq1$ such that $(M^n)_{ij}>0$.  We define equivalent classes $C_1,C_2,\ldots$ such that, for each $k$, if $i \in C_k$, then $j \in C_k$ if and only if $i\rightarrow j$ and $j\rightarrow i$, for all $j$. This induces a partition of the set of types $\mathcal{S}$ and we write that $C_k\rightarrow C_\ell$  when there exist $i\in C_k$ and $j\in C_\ell$ such that $i\rightarrow j$. 
We denote by  $M_{k}$ the irreducible mean progeny matrix  restricted to types in $C_k$, that is, $M_{k}=(M_{ij})_{i,j\in C_k}$, and by  $\nu_k$  the convergence norm of $M_{k}$.
 
\begin{prop}\label{locextcrit3} If $i$ is a type in $C_k$, then the partial extinction probability $\tilde{q}_i$ is strictly less than 1 if and only if $\nu_k > 1$ or  there exists a class $C_\ell$ such that $C_k\rightarrow C_\ell$ and $\nu_\ell>1$.\end{prop} 

\begin{proof}
Let $i\in C_k$ and assume that $\nu_k>1$. Then, by Proposition~\ref{locextcrit}, the probability that every type in $C_k$ eventually becomes extinct, given that the initial type is in $C_k$, is strictly less than one, hence $\tilde{q}_i<1$.

Now assume that $i\in C_k$, and that there exists a class $C_\ell$ such that $C_k\rightarrow C_\ell$ and $\nu_\ell>1$. Then, $i\rightarrow j$ for all $j\in C_\ell$, that is, there is a positive probability that type $i$ has type $j\in C_\ell$ among its descendants; moreover, since $\nu_\ell>1$, starting from any $j\in C_\ell$, the probability that every type in $C_\ell$ eventually becomes extinct is strictly less than one by Proposition~\ref{locextcrit}. We thus obtain $\tilde{q}_i<1$. 

If $\nu_k\leq 1$ and there is no class $C_\ell$ such that $C_k\rightarrow C_\ell$ and $\nu_\ell>1$, then all classes $C_\ell$ such that $C_k\rightarrow C_\ell$ satisfy $\nu_\ell\leq1$. In other words, by Proposition~\ref{locextcrit}, all the descendants of type $i$ will generate a process which partially becomes extinct with probability 1. So partial extinction is almost sure if the process is initiated by type $i$, and we have $\tilde{q}_i=1$.
\end{proof}

\subsection{Global extinction --- $M$ irreducible}
	\label{subsec:globalcriteria}

We assume again that $M$ is irreducible.  By  Lemma \ref{lem1} and Proposition \ref{locextcrit}, we know that if  $\nu >1$, then ${\vc q}= {\tilde{\bs{q}}}  <\vc1$, and
 if $\nu \leq 1$, then
$\vc q \leq {\tilde{\bs{q}}} =\vc1$.  One question which remains is to determine additional conditions that guarantee $\vc q = {\tilde{\bs{q}}} =\vc1$.  

The most precise answers are conditioned on the dichotomy property, which states that with probability 1 the population either becomes extinct or drifts to infinity (Harris~\cite{harris63}). In the finite case, this follows under very mild conditions but it is more problematic if the number of types is infinite. In particular,
%
%
Tetzlaff~\cite[Condition~2.1 and Proof of~Proposition~2.2]{tetzlaff03} gives the following sufficient condition for the dichotomy property to hold: it suffices that for all $k \geq 1$, there exists an index $m_k$ and a real number $d_k > 0$ such that 
\begin{align} \label{dich} 
\inf_{i} \mathds{P}[|\mathcal{Z}_{m_k}| = 0 \; | \;\varphi_0=i,\,1 \leq |\mathcal{Z}_1| \leq k] \geq d_k.
\end{align} 
%
This indicates that there is a positive, and bounded away from zero, probability for the population to become extinct.  The next property is proved in~\cite{tetzlaff03}.
\begin{prop} \label{prop:Tetzlaff} 
Assume that the dichotomy property holds.  If the limit $\liminf_{n\rightarrow\infty} (M^n\,\vc1)$ is finite, then ${\vc q}=\vc 1$.
\end{prop}
A direct consequence, which brings the convergence norm back into the picture, is the following.
\begin{prop}\label{NC}
Assume that the dichotomy property holds. If there exist $\lambda\leq 1$ and $\vc x>\vc 0$ such that $\vc x\vc1<\infty$ and $\vc x M\leq\lambda \vc x$, then ${\vc q}=\vc 1$.
\end{prop}

\begin{proof}
Under the assumptions of the proposition, $\vc x M^n\vc1 \leq \lambda^n \vc x\vc1$, which implies that $\lim_{n\rightarrow\infty}\vc x M^n\vc1 < \infty$. Applying Fatou's Lemma, we obtain 
$
\vc x \lim_{n\rightarrow\infty}(M^n\vc1)< \infty,
$
which leads to $\lim_{n\rightarrow\infty}(M^n\vc1)< \infty$ since $\vc x > \vc 0$. Thus, by Proposition~\ref{prop:Tetzlaff} the result follows.
\end{proof}

If such a $\lambda$ exists, it is at least equal to $\nu$, and we remember that $\nu \leq 1$ is a necessary condition for $\vc q = \bs{1}$.  The difference between $\lambda$ and $\nu$, and the additional constraint imposed by this proposition, is that the measure associated to $\lambda$ must be convergent, which is not necessarily the case with the measure associated with $\nu$.

\subsection{Growth rate and extinction}
	\label{sec:initial}

When the number of types is finite and the process is irreducible, the expected total population size increases, or decreases, asymptotically geometrically: independently of the initial type,  $\mathbb{E}[|\mathcal{Z}_n|]\sim \rho^n$ where $\rho$ is the spectral radius of $M$.  This is no longer the case when the number of types is infinite, and the evolution of $\mathbb{E}[|\mathcal{Z}_n|]$ may depend on the distribution of $\varphi_0$.  Actually, it is possible for a process to become globally extinct almost surely while the expected population size increases without bounds.  This  we show below,  and we give one example in the next section.

Assume that there exists a probability measure $\vc\alpha_1$ such that $\vc\alpha_1 M \leq \lambda_1 \vc\alpha_1$ with $\lambda_1 < 1$, and a probability measure $\vc\alpha_2$ such that 
$\vc\alpha_2 M \geq \lambda_2 \vc\alpha_2$ with $\lambda_2 > 1$.   If, in addition, the dichotomy property holds, then $\vc q = \vc 1$ by Proposition~\ref{NC}.

If $\varphi_0$ has distribution $\vc\alpha_1$, then 
\begin{align*}
\mathbb{E}[|\mathcal{Z}_n|] = \vc \alpha_1 \,M^n\,\vc 1\leq \lambda_1^n
\end{align*}
so that  $\lim_{n \rightarrow \infty}\mathbb{E}[|\mathcal{Z}_n|] = 0$.
In the contrary, if $\varphi_0$ has the distribution $\vc\alpha_2$, then by a similar argument $\lim_{n \rightarrow \infty}\mathbb{E}[|\mathcal{Z}_n|] = \infty$ and the extinction probability is equal to $\vc\alpha_2 \vc q = 1$.

\section{Illustration}
      \label{sec:rand}

We illustrate the results of the previous sections with two examples, one for which $M$ is tridiagonal (and the process is irreducible) and one for which $M$ is super-diagonal (and the process is reducible).  

\subsection{Irreducible tridiagonal case}

This example corresponds to a homogeneous branching random walk on positive integers with a reflecting wall at $z=1$.  The mean progeny matrix~is
\begin{align} \label{trid}
M& = \left[\begin{array}{cccccc} 
b & \;\; c & \\
a & \;\; b & \;c \\ 
   & \;\; a & \;b & c \\
    &   & \ddots & \ddots & \ddots 
\end{array}\right], 
\end{align} 
where $a$ and $c$ are strictly positive, and $b$ is nonnegative.
\begin{prop} \label{propertyM} Assume that $M$ is as given in \eqref{trid}. Then, its convergence norm is ${\nu} = b + 2\sqrt{ac}$, and there exists $\vc x > \vc 0$ such that $\vc x M=\lambda \vc x$ for all $\lambda \geq\nu$. In addition, $\vc x\vc1<\infty$ if and only if $\lambda\in[\nu,a+b+c)$ and $a>c.$ 

The strictly positive and convergent invariant measure $\vc x$ is given by
\begin{align}
x_k & = \eta \,k\,({\sqrt{ac}}/{a})^k  && \mbox{ if } \lambda=\nu, \label{eqn:xkeq} \\
       & = \eta \,\{[({\lambda - b +\sqrt{\Delta}})/({2a})]^k-[(\lambda - b -\sqrt{\Delta})/(2a)]^k\} && \mbox{ if } \lambda>\nu, \label{eqn:xkiq} 
\end{align}
for $k\geq1$, 
where $\eta$ is an arbitrary constant and $\Delta = (b - \lambda)^2 - 4ac$. 
\end{prop}

\begin{proof}
Let $M^{(k)}$ denote the $k \times k$ north-west truncations of $M$. Then, by a modification of van Doorn \emph{et al.} \cite[Theorem~1, Eqn~(9)]{dfz09} we obtain 
\begin{align*}
\spc(M^{(k)}) = \min_{\bs{u} \geq \bs{0}}  \max_{1 \leq i \leq k} \left\{ b + u_{i + 1} + \frac{ac}{u_i} \right\} .
\end{align*} 
Then, by Seneta~\cite[Theorem~6.8]{senetabook} 
\begin{align*}
\nu & = \lim_{ k \rightarrow \infty} \spc(M^{(k)})  =  \min_{\bs{u} \geq \bs{0}}  \sup_{i} \left\{ b + u_{i + 1} + \frac{ac}{u_i}  \right\} =  b + 2\sqrt{ac},
\end{align*} 
with the last equality following from arguments analogous to those in the proof of Theorem~3.2 in Latouche \emph{et al.} \cite{latouche11}.

Now, for any $\bs{x}$ to satisfy $\bs{x} M = \lambda \bs{x}$, its elements have to satisfy the constraints 
\begin{align}
bx_1  + ax_2 & = \lambda x_1  \label{eqn:con1} \\\nonumber
ax_{k + 1}  + (b - \lambda)x_k  + cx_{k - 1} & = 0 \quad \mbox{for } k \geq 2. 
\end{align} 
Let $\Delta = (b - \lambda)^2 - 4ac$. There are three cases, for each of which $\bs{x}$ takes a specific form (Korn and Korn~\cite[Chapter~20, Section~4.5]{kornkorn61}). 

\textbf{Case 1:} $\Delta = 0$. Then, $\lambda = b \pm 2\sqrt{ac}$, and for $k \geq 1$,
\begin{align}
x_k = c_1[(\lambda - b)/(2a)]^k + c_2 k [(\lambda - b)/(2a)]^k, \label{eqn:xkCase1}
\end{align}  
where $c_1$ and $c_2$ are constants. Substituting \eqref{eqn:xkCase1} into \eqref{eqn:con1} gives us $c_1 = 0$. To ensure $x_k > 0$ for all $k$, it is necessary that $\lambda > b$. Consequently, $\lambda = b + 2\sqrt{ac}$, and we obtain \eqref{eqn:xkeq}. For $\bs{x}$ to be convergent, we require that $\sqrt{ac}/a < 1$ and thus $a > c$. 

\textbf{Case 2:} $\Delta > 0$. Then, $\lambda < b - 2\sqrt{ac}$ or $\lambda > b + 2\sqrt{ac}$, and for $k \geq 1$,
\begin{align} \label{eqn:xkCase2} 
x_k = c_3[((\lambda - b + \sqrt{\Delta})/(2a)]^k + c_4[((\lambda - b - \sqrt{\Delta})/(2a)]^k, 
\end{align} 
where $c_3$ and $c_4$ are constants. Substituting \eqref{eqn:xkCase2} into \eqref{eqn:con1} gives us $c_3 = -c_4$, and \eqref{eqn:xkCase2} simplies to \eqref{eqn:xkiq}. It is clear from \eqref{eqn:xkiq} that $\bs{x} > \bs{0}$ if and only if $\lambda > b$.  Thus, $\lambda > b + 2\sqrt{ac}$. 

Finally, $\bs{x} \bs{1} < \infty$ if and only if $((\lambda - b) + \sqrt{\Delta})/(2a) < 1$, the latter being equivalent to $\lambda < 2a + b$ and $\lambda < a + b + c$. As $b + 2\sqrt{ac} < \lambda < 2a + b$, both $a < c$ and $a = c$ lead to a contradiction. Consequently, $a > c$.

\textbf{Case 3:} $\Delta < 0$. Then, $b - 2\sqrt{ac} < \lambda < b + 2\sqrt{ac}$ and 
\begin{align}
x_k = (c/a)^k(c_5 \cos (k \phi) + c_6 \sin (k \phi)), \label{eqn:xkCase3}
\end{align} 
where $\phi = \arccos(b/(2\sqrt{ac})), 0 < \phi < \pi$ and $c_5$ and $c_6$ are constants. 

Since we are looking for $\bs{x} > \bs{0}$, Case 3 is not feasible. Indeed, we can rewrite \eqref{eqn:xkCase3} as, for $k \geq 1$, $x_k = c_7 (c/a)^k \cos(c_8 + k \phi)$ where $0 \leq c_8 < 2 \pi$ and $c_7$ is arbitrary.  It can be easily shown that there exists $k_0$ such that $\cos c_8$ and $\cos (c_8 + k_0 \phi)$ have different signs. 
\end{proof}

Among the progeny distributions that may be associated with the mean progeny matrix given in \eqref{trid}, we choose 
\begin{align*}
P_i(\vc s) & =  ({b}/{t}) s_i^t + ({c}/{t})s_{i+1}^t+1-({b+c})/{t} &&  \mbox{for }  i=1,\\
& = ({a}/{u})s_{i-1}^u + ({b}/{u}) s_i^u + ({c}/{u}) s_{i+1}^u+1- ({a+b+c})/{u} && \mbox{for } i\geq 2,
\end{align*}
where $t=\lceil b+c \rceil+1$ and $u=\lceil a+b+c \rceil+1$. By varying $a,b,$ and $c$, we shall cover the three possible cases $\vc q=\tilde{\vc q}<\vc1$, $\vc q=\tilde{\vc q}=\vc1$, and $\vc q<\tilde{\vc q}=\vc1$.  

\paragraph{Case 1: $\vc q=\tilde{\vc q}<\vc1$.} Take  $a= b=1/2$, and $c=1/3$. With these, $\nu=1.28>1$, and $\vc q=\tilde{\vc q}<\vc1$ by Lemma \ref{lem1} and Proposition \ref{locextcrit}.

We illustrate in Figure~\ref{fig1} the convergence of the sequences $\{\vc q^{(k)}\}$ and $\{\tilde{\vc q}^{(k)}\}$.  On the left, we plot $q_1^{(k)}$ and  $\tilde{q}_1^{(k)}$ for $k= 1$ to 20; the two sequences rapidly converge to a common value approximately equal to 0.89. 
On the right, we plot $q_i^{(20)}$ and  $\tilde{q}_i^{(20)}$ for $i = 1$ to 20; we observe  that the first 15 entries are well-approximated after 20 iterations but the next entries require more iterations because for high values of $i$, the approximation process for $q_i$ and $\tilde{q}_i$ starts with a higher value of $k$. 

A sequence $\{x_k\}_{k\geq0}$ converges linearly to $x$ if there exists $0<\mu<1$ such that 
$ \lim_{k\rightarrow\infty} {|x-x_{k+1}|}/{|x-x_k|}=\mu$, and $\mu$ is called the convergence rate. Our numerical investigations indicate that the convergence of $q_i^{(k)}$ as well as that of  $\tilde{q}_i^{(k)}$ is linear, for fixed $i$.  We give one example in Figure~\ref{fig2}, where we plot the ratios $|q_1 - q_1^{(k+1)}|/|q_1 - q_1^{(k)}|$ and
$|\tilde q_1 - \tilde q_1^{(k+1)}|/|\tilde q_1 - \tilde q_1^{(k)}|$; not knowing the value of either $q_1$ or $\tilde q_1$, we have used the values at the 20th iteration. The two sequences are seen to converge linearly at the same rate $\mu = 0.26$ approximately.

\begin{figure}[!tbp]%
\centering 
\includegraphics[scale=0.3]{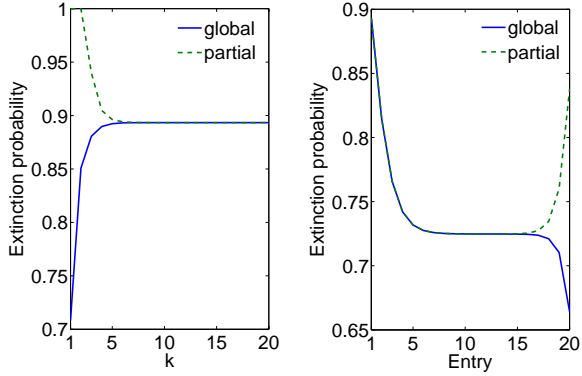}
\caption{Case 1. $a=b=1/2, $ and $c=1/3$. Left: the values of $q_1^{(k)}$ (continuous line) and of $\tilde{q}_1^{(k)}$ (dashed line).  Right: first entries of $\vc q^{(20)}$ (continuous line) and $\tilde{\vc q}^{(20)}$ (dashed line).}
\label{fig1}
\end{figure}

\begin{figure}[!tbp]%
\centering
\includegraphics[scale=0.3]{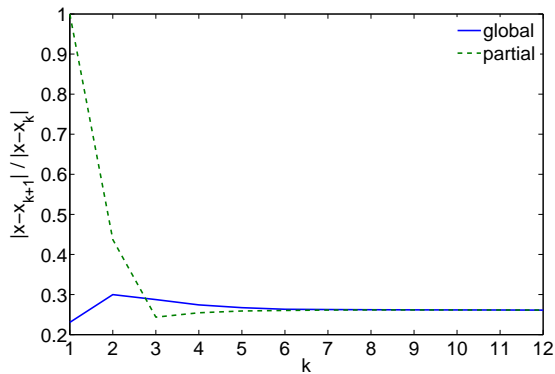}
\caption{Case 1. $a=b=1/2,$ and $c=1/3$. Convergence rates of the sequences $q_1^{(k)}$ (continuous line) and $\tilde{ q}_1^{(k)}$ (dashed line).}
\label{fig2}
\end{figure}

\paragraph{Case 2:  $\vc q = \tilde{\vc q} = \vc 1$} The parameters here are  $a=b=1/2$, and $ c=1/25$. Here ${{a>c}}$ and for any individual, the type of its descendants drifts over successive generations toward type 1, the least prolific of types.  The convergence norm is  $\nu=0.78<1$, which implies that $\tilde{\vc q}=\vc 1$.  We shall conclude from Proposition~\ref{propertyM} and Proposition~\ref{NC} (with $\lambda=\nu$) that $\vc q = \vc 1$ as well, once we show that the dichotomy property holds.  The progeny generating function is given by
\begin{align} 
   \label{e:pi}
P_i(\vc s) & = (1/4) s_i^2 + (1/50)s_{i+1}^2+(73/100) && \mbox{ for } i=1,\\
   \label{e:pia}
                & = (1/6) s_{i-1}^3 +(1/6) s_i^3 + (1/75) s_{i+1}^3+(49/75) && \mbox{ for } i\geq 2.
\end{align}
To verify that the dichotomy property holds, we use the sufficient condition~(\ref{dich}). 
In view of (\ref{e:pi}, \ref{e:pia}), we observe that for all $i$, $\mathds{P}[|\mathcal{Z}_{2}| = 0 \; | \;\varphi_0=i,\,1 \leq |\mathcal{Z}_1| \leq k]\geq (\min(73/100,49/75))^k,$ and we conclude that  \eqref{dich} is satisfied with $m_k=2$ and $d_k=(49/75)^k$.

\begin{figure}[!tbp]
\begin{center}\includegraphics[scale=0.3]{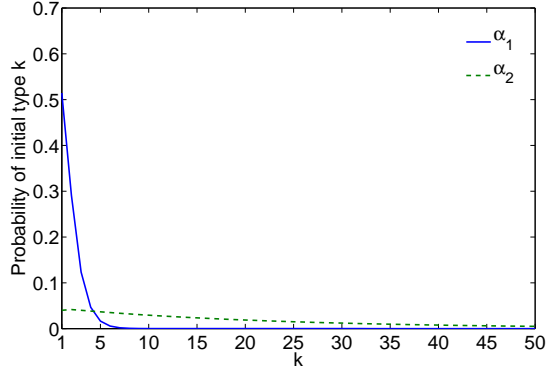}\end{center}
\caption{Case 2: $a= b=1/2$, $c=1/25$. The first $50$ components of the initial type distribution vector ${\vc\alpha}_1$, for which $\lim_n \mathbb{E}[|\mathcal{Z}_n|]=0$, and of  ${\vc\alpha}_2$, for which $\lim_n \mathbb{E}[|\mathcal{Z}_n|]=\infty$.}
\label{init}
\end{figure}

To illustrate the observation made in Subsection~\ref{sec:initial} about the effects of the initial type's distribution, we take the parameters 
\[
 \lambda_1=\nu=0.78<1 \qquad \mbox{and} \qquad \lambda_2=1.02<1.04 \;(=a+b+c).
\]
By Proposition~\ref{propertyM}, there exist $\vc\alpha_1>0$ and $\vc\alpha_2>0$ such that $\vc\alpha_1M=\lambda_1\vc\alpha_1$ and $\vc\alpha_1\vc 1=1$, and such that $\vc\alpha_2M=\lambda_2\vc\alpha_2$ and   $\vc\alpha_2\vc 1=1$. 

If $\varphi_0$ has the distribution $\vc\alpha_1$, then $\lim_n \mathbb{E}[|\mathcal{Z}_n|]=0$, while if it has the distribution  $\vc\alpha_2$, then $\lim_n \mathbb{E}[|\mathcal{Z}_n|]=\infty$.  In both cases, extinction is with probability 1.  We plot in Figure~\ref{init} the first 50 components of $\vc\alpha_1$ and $\vc\alpha_2$.  The difference between the two is that the distribution $\vc\alpha_1$ is concentrated on small types, so that the process has less chance of building a high population before its eventual extinction. 

\paragraph{Case 3: $\vc q  \leq \tilde{\vc q} = \vc 1$.} 
Take $a=1/25,  b= c=1/2$.
Here, $a <c $ and $\nu=0.78<1<a+b+c$; thus, $\vc q\leq\tilde{\vc q}=\vc 1$ but we do not know if $\vc q =\vc 1$ or not. 

We show on  the left of Figure~\ref{fig4} the values of $q_1^{(k)}$ and $\tilde q_1^{(k)}$ for $k=1$ to 60.  Judging from this, we conclude that $q_1 < 1 = \tilde q_1$.   On the right of that figure, we give $q_i^{(60)}$ and $q_i^{(60)}$ for $i=1$ to 60.

\begin{figure}[!tbp]
\begin{center}\includegraphics[scale=0.3]{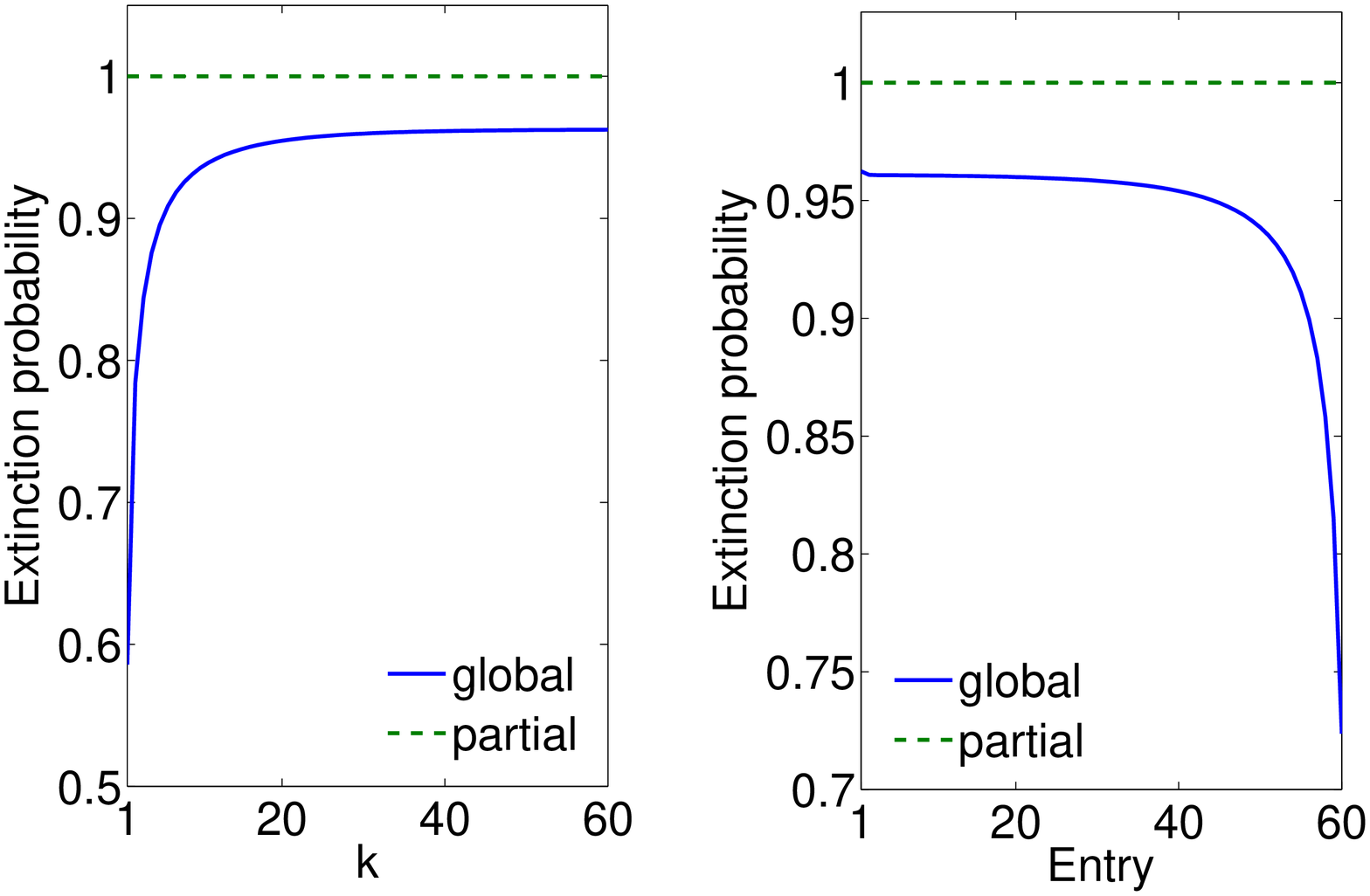}\end{center}
\caption{Case 3: $a=1/25,  b=c=1/2$. 
Left: the values of $q_1^{(k)}$ (continuous line) and of $\tilde{q}_1^{(k)}$ (dashed line).  Right: first entries of $\vc q^{(60)}$ (continuous line) and $\tilde{\vc q}^{(60)}$ (dashed line).}
\label{fig4}
\end{figure}

\begin{figure}[!tbp]
\begin{center}\includegraphics[scale=0.3]{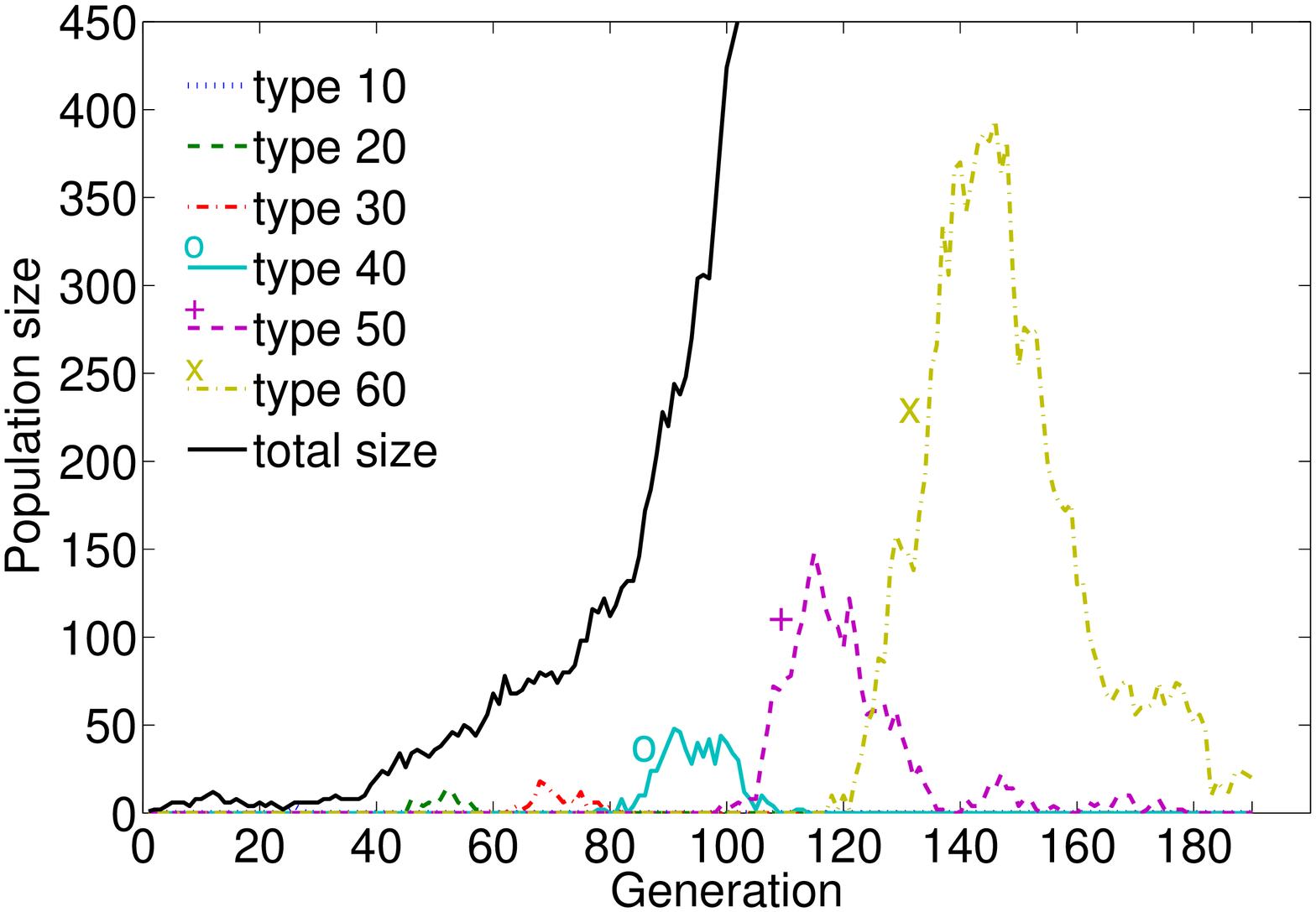}\end{center}
\caption{Case 3. $a=1/25,  b=c=1/2$. Simulation of the evolution of the population size in different types and of the total population size}
\label{fig5}
\end{figure}

To confirm the conclusion that  $\vc q < \tilde{\vc q} = \vc 1$, we have simulated the branching process and we give one particular sample path on
Figure~\ref{fig5}: the whole population $|\calz_n|$ seems to grow without bounds, while individual types appear, grow in importance, and eventually disappear from the population.

\subsection{Reducible example}

Consider the mean progeny matrix with the structure
\begin{align*} 
M& = \left[\begin{array}{cccccc} 
b_1 & \;\; c_1 & \\
 & \;\; b_2 & \;c_2 \\ 
   & \;\;  & \;b_3 & c_3 \\
    &   &  & \ddots & \ddots 
\end{array}\right],
\end{align*} 
where $b_i\geq 0$ and $c_i > 0$ for all $i$.

In this special case of reducible mean progeny matrix we may associate another interpretation to the sequence $\{\tilde{\vc q}^{(k)}\}$.  Let us define the \emph{local extinction} of a specific type. This event is $E_k = [\lim_{n\rightarrow\infty} Z_{nk}=0]$, independently of the other types.

A moment of reflection shows that $E_k\equiv \cap_{\ell\leq k} E_\ell$
and, furthermore, that $\tilde{q}^{(k)}_i$ is the probability that
type $k$ eventually becomes extinct, given that the process starts
with a first individual of type $i$.  This allows us to give another
proof that $\tilde{\vc q}^{(k)}\geq \tilde{\vc q}^{(k+1)}$ and that the sequence  converges to $\tilde{\vc q}$:
\begin{align*} \lim_{k\rightarrow\infty} \tilde{\vc q}^{(k)} & = \lim_{k\rightarrow\infty}\mathds{P}[E_k\,|\,\varphi_0] \  = \lim_{k\rightarrow\infty}\mathds{P}[\cap_{\ell\leq k} E_\ell\, |\,\varphi_0] \\
& = \mathds{P}[\cap_{\ell\leq \infty} E_\ell\,\Big|\,\varphi_0] \  =\tilde{\vc q}.
\end{align*}
In the reducible case, the equation $\vc s=\vc P(\vc s)$ may have more than two distinct solutions and, in particular, it is possible that \mbox{$\vc q<\tilde{\vc q}< \vc 1$}, as we show on one example.

Take $b_i=0$ and $c_i=1.9$ for every $i$ except for $i = 10$, where $b_{10}=1.6$ and $c_{10}=0.8$. That is, in general, type $i\neq10$ individuals have only children of the next type,  slightly less than two on average, and
type 10 is different. If it were not for type 10, the whole population would behave as a supercritical process, with each type getting extinct after one generation. Individuals of type 10 do reproduce themselves, in a supercritical fashion.

Assume that the progeny generating function is
\begin{align*} 
P_i(\vc s) & = (19/30)s_{i+1}^3+(11/30) && \mbox{ for } i\neq 10,\\
                & = (2/5) s_i^4 + (1/5) s_{i+1}^4+(2/5)  && \mbox{ for } i=10.
\end{align*} 
As the sequence $\{\spc(M^{(k)})\}$ converges to $\bar{\nu}=1.6>1$, we know by Proposition~\ref{locextcrit2} that $\tilde{\vc q} \lneqq\vc 1$. Furthermore,  Proposition~\ref{locextcrit3} implies that $\tilde{q}_i<1$ for $1\leq i\leq 10$, and $\tilde{q}_i=1$ for $i\geq 11$.

We show $\{q_8^{(k)}\}$ and $\{\tilde{q}_8^{(k)}\}$
on the left in  Figure~\ref{fig6} and the plot  clearly makes it appear that $q_8 < \tilde q_8 < 1$.  On the right, we give the values of $q_i^{(30)}$ and $\tilde q_i^{(30)}$ for $1 \leq i \leq 30$.  For $i \geq  11$, local extinction has probability 1 since every type exists for one generation only, and the global probability, at least if $i$ is sufficiently smaller than 30, is close to 0.41, the extinction probability of a single-type branching process with 
 progeny generating function 
\begin{align*} 
P(s) & = (19/30)s^3+(11/30).
\end{align*} 
We thus see that if extinction happens in the single-type process, then it does so in a few generations.
%
\begin{figure}
\begin{center}\includegraphics[scale=0.30]{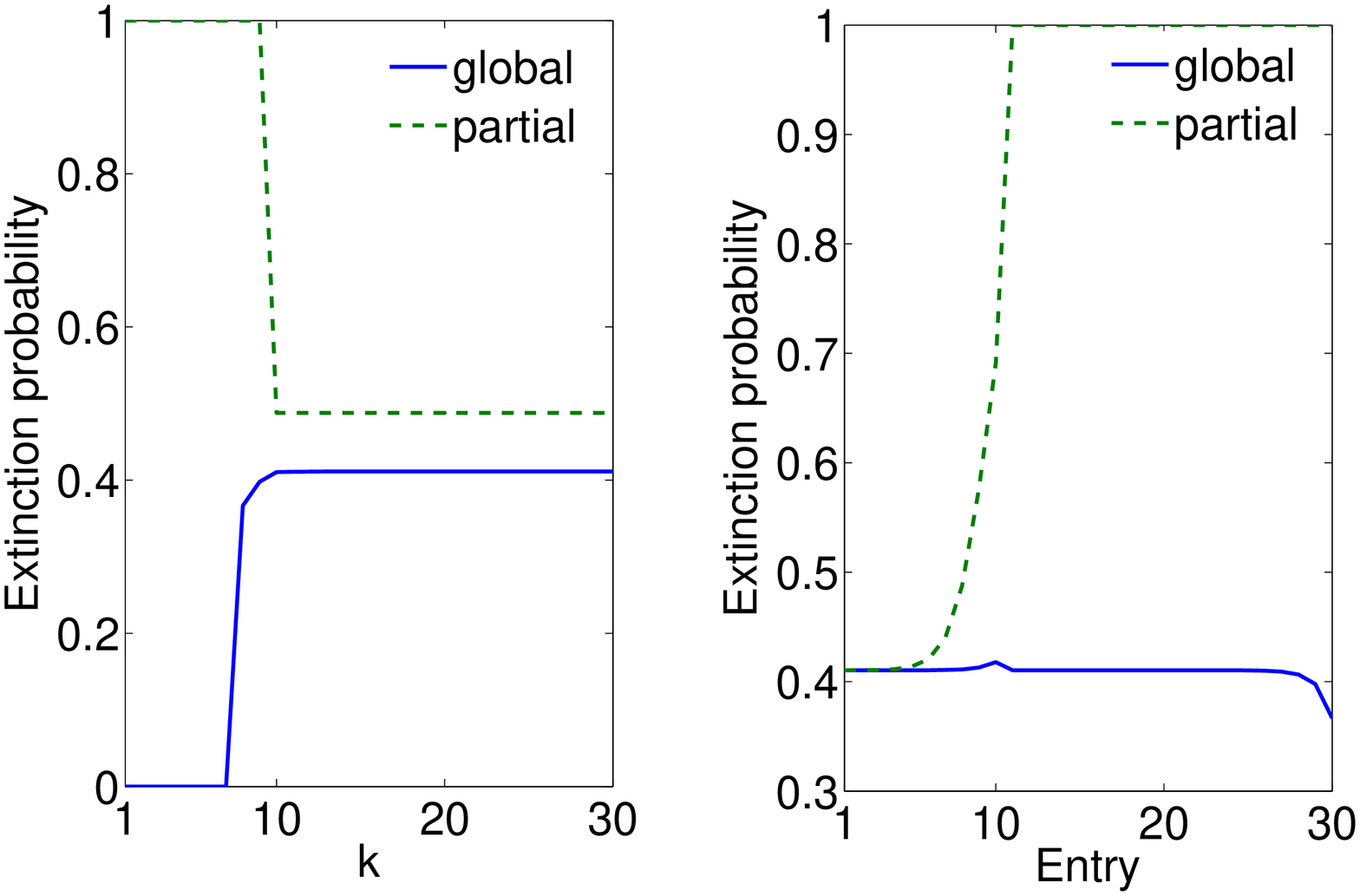}\end{center}
\caption{Left: the values of $q_8^{(k)}$ (continuous line) and of $\tilde{q}_8^{(k)}$ (dashed line).  Right: first entries of $\vc q^{(30)}$ (continuous line) and $\tilde{\vc q}^{(30)}$ (dashed line).}
\label{fig6}
\end{figure}

%
%
%
%
%
%
%



\acks

All three authors thank the Minist\`ere de la Communaut\'e fran\c{c}aise de
Belgique for funding this research through the ARC grant AUWB-08/13--ULB~5.  
 The first and third authors also acknowledge the financial support of the Australian Research Council through the Discovery Grant DP110101663. 

%
%
%
%

\bibliographystyle{apt}
\bibliography{BP_infinity_bib} 

%
%
%
%

\end{document}